\documentclass[11pt, a4paper]{amsart}
\usepackage{amssymb, array, amsmath, amscd, pdfpages, enumerate, amsthm, fixltx2e, setspace, hyperref}

\DeclareMathOperator*{\Ran}{Ran}
\DeclareMathOperator*{\Ker}{Ker}

\DeclareMathOperator*{\rank}{rank}

\newcommand{\dd}{\mathrm{d}}

\newcommand{\B}{\mathcal{B}}

\newcommand{\T}{\mathbb{T}}

\newcommand{\CC}{\mathbb{C}}
\newcommand{\NN}{\mathbb{N}}
\newcommand{\ZZ}{\mathbb{Z}}
\newcommand{\DD}{\mathbb{D}}

\newtheorem{thm}{Theorem}[section]
\newtheorem{prp}[thm]{Proposition}
\newtheorem{lem}[thm]{Lemma}
\newtheorem{cor}[thm]{Corollary}
\theoremstyle{definition}
\newtheorem{rem}[thm]{Remark}

\numberwithin{equation}{section}

\begin{document}

\title[Asymptotic behaviour of coupled systems]{Asymptotic behaviour of coupled systems in discrete and continuous  time}
\author{Lassi Paunonen}
\address{Department of Mathematics, Tampere University of Technology, PO.\ Box 553, 33101 Tampere, Finland}
\email{lassi.paunonen@tut.fi}

\author{David Seifert}
\address{St John's College, St Giles, Oxford\;\;OX1 3JP, United Kingdom}
\email{david.seifert@sjc.ox.ac.uk}

\begin{abstract}
This paper investigates the asymptotic behaviour of solutions to certain infinite systems of coupled recurrence relations. In particular, we obtain a characterisation of those initial values which lead to a convergent solution, and for initial values satisfying a slightly stronger condition we obtain an optimal estimate on the rate of convergence. By establishing a connection with a related problem in continuous time, we are able to use  this optimal estimate to improve the rate of convergence in the continuous setting obtained by the authors in a previous paper.  We illustrate the power of the general approach by using it to study several concrete examples, both in continuous and in discrete time.
 \end{abstract}

\thanks{This work was carried out while the first author visited Oxford in March 2016.  The visit was funded by the   COST Mathematics for industry network.}

\maketitle

\section{Introduction}\label{sec:intro} 

Consider a situation in which there are countably many agents, indexed by the integers $\ZZ$, such that agent $k\in\ZZ$ at time $n\ge0$ is in the position $x_k(n)\in\CC$. Suppose that the agents' positions change at each time step according to the rule
\begin{equation}\label{eq:bugs_intro}
x_k(n+1)=(1-\alpha)x_k(n)+\alpha x_{k-1}(n),\quad k\in\ZZ,\; n\ge0,
\end{equation}
where $\alpha\in(0,1)$ is a fixed constant. Thus agent $k$ changes its position at each time step by moving from its current position a fraction $\alpha$ of its current separation from agent $k-1$  in the direction of agent $k-1$. The purpose of this paper is to develop general techniques which allow one to study the asymptotic behaviour of solutions to the above system and similar more complicated ones. The main questions of interest are (i) which initial constellations of the agents will lead to convergence of the overall system to an equilibrium point in a suitable sense, (ii) what is the equilibrium when it exists and (iii) at what rate does the convergence take place?

In order to be able to answer these questions in a unified manner and also extend our conclusions to a broader class of examples, we consider the more general recurrence relation 
\begin{equation}\label{rec}
x_k(n+1)=T_0x_k(n)+T_1 x_{k-1}(n),\quad k\in\ZZ,\; n\ge0,
\end{equation}
where  $x_k(n)\in\CC^m$ for some given positive integer $m$ and where $T_0$, $T_1$ are $m\times m$ matrices satisfying certain assumptions to be spelled out in due course. We proceed in the main part of this paper by considering a single operator $T$ which acts on suitable spaces of sequences $(x_k)$, indexed by $k\in\ZZ$, as $T(x_k)=(T_0 x_k+T_1 x_{k-1})$. Crucial among the assumptions we make is that there exists a rational function $\phi_T$, the so-called \emph{characteristic function}, such that 
$$T_1R(\lambda,T_0)T_1=\phi_T(\lambda)T_1$$
for all $\lambda\in\CC$ such that the resolvent operator $R(\lambda,T_0)=(\lambda-T_0)^{-1}$ exists. It is this assumption that makes our systems tractable even when the matrices $T_0$ and $T_1$ do not commute, by allowing for a very precise spectral analysis of the operator $T$ which in turn leads to sharp estimates for asymptotic behaviour of its powers $T^n$, $n\ge0$. As we shall see, a characteristic function exists for a broad class of systems including all of the applications we have in mind. In fact, in the cases of particular interest to us the characteristic function $\phi_T$ is of a specific and rather simple form, and  a considerable part of the paper focusses on this important special case.

The general approach taken in this paper can be viewed as a discrete counterpart to the authors' previous paper \cite{PauSei15}, in which the corresponding continuous-time problem is studied. As it turns out, the discrete setting presents its own challenges and requires new techniques but, in return, leads to optimal results which can even be used to improve the known results in the continuous setting. Indeed, this latter fact is one of our main motivations for studying discrete systems even though they are natural and interesting in their own right. Our paper can therefore be viewed as a contribution to the broader study of so called \emph{spatially invariant systems}; see for instance \cite{BamPag02}. More specifically, the introductory example presented above can be viewed as a discrete counterpart of the so-called \emph{robot rendezvous problem} studied in \cite{FeiFra12,FeiFra12b}, while the main motivating examples for more sophisticated cases of the general model arise in the study of so-called \emph{platoon models}; see for instance  \cite{PloSch11, SwaHed96,ZwaFir13}. For related works in the study of multi-agent systems in discrete and continuous time, see for instance~\cite{LinMor07,WenUgr13, PoAnt09report}.

The paper is organised as follows. In Section~\ref{sec:gen} we present  the general operator-theoretic results required to study our class of systems, culminating in Theorem~\ref{thm:asymp_disc}, which gives a complete description of those initial constellations leading to convergent solutions, shows how the limit (when it exists) is related to the initial constellation and furthermore gives an estimate for the rate of convergence for certain initial constellations. In particular, the result answers in the general setting all three questions raised above in the context of the toy model \eqref{eq:bugs_intro}. In Section~\ref{sec:ex} we show explicitly how the general result can be applied both to this simple example and also to a more complex one in which the agents' state vectors consist not only of their positions but involve also a velocity component. In Section~\ref{sec:cont} we provide a link between the discrete and the continuous settings and show how Theorem~\ref{thm:asymp_disc} can be used to improve the main result of \cite{PauSei15} in an important special case. Finally, in Section~\ref{sec:cont_ex} we apply this improved result to give sharper rates of decay in the platoon model studied in \cite{PloSch11, SwaHed96,ZwaFir13}.

The notation we use is more or less standard throughout. Thus, given a complex Banach space $X$, the norm on $X$ will typically be denoted by $\|\cdot\|_X$ or simply by $\|\cdot\|$. In particular, for $m\in\NN$ and $1\le p\le\infty$, we let $\ell^p(\CC^m)$ denote the space of doubly infinite sequences $(x_k)$ such that $x_k\in\CC^m$ for all $k\in\ZZ$ and $\sum_{k\in\ZZ}\|x_k\|^p<\infty$ if $1\le p<\infty$ and $\sup_{k\in\ZZ}\|x_k\|<\infty$ if $p=\infty$. Here and in all that follows we endow the finite-dimensional space $\CC^m$ with the standard Euclidean norm and we consider $X=\ell^p(\CC^m)$ with the norm given for $x=(x_k)$ by $\|x\|_X=(\sum_{k\in\ZZ}\|x_k\|^p)^{1/p}$ if $1\le p<\infty$ and $\|x\|=\sup_{k\in\ZZ}\|x_k\|$ if $p=\infty$. With respect to this norm $X$ is a Banach space for $1\le p\le\infty$ and a Hilbert space when $p=2$. We write $\B(X)$ for the space of bounded linear operators on $X$, and given $T\in\B(X)$ we write $\Ker(T)$ for the kernel and $\Ran(T)$ for the range of $A$. Moreover, we let $\sigma(T)$ denote the spectrum of $T$ and $\rho(T)=\CC\setminus\sigma(T)$ the resolvent set of $T$.  We write $\sigma_p(T)$ for the point spectrum and $\sigma_{ap}(T)$ for the approximate point spectrum of $T$.  For $\lambda\in\rho(T)$ we write $R(\lambda,T)$ for the resolvent operator $(\lambda-T)^{-1}$. Asymptotic notation, such as $O$, $o$ and $\asymp$, is used in the usual way.   Finally, we denote by $\DD$ the open unit disc $\{\lambda\in\CC:|\lambda|<1\}$.

\section{The discrete-time system}\label{sec:gen} 
 
We begin by introducing the general system to be studied.  Given $p$ with $1\le p\le\infty$, let $X=\ell^p(\CC^m)$.    We may write \eqref{rec} together with an initial condition in the form
\begin{equation}
\label{CP_disc} 
\begin{cases}
x(n+1)= Tx(n), \quad n\ge0,\\
 x(0)=x_0\in X,
\end{cases}
\end{equation}
where $x(n)=(x_k(n))$ and $Tx=(T_0x_k+T_1x_{k-1})$ for all $x=(x_k)\in X$. We assume in what follows that 
\begin{equation}\label{char}
T_1 R(\lambda,T_0)T_1=\phi_T(\lambda)T_1,\quad \lambda\in\rho(T_0),
\end{equation}
for some rational function $\phi_T:\rho(T_0)\to\CC$ which we call the \emph{characteristic function} of our system. The existence of a characteristic function is crucial to all that follows and will reduce several key questions about the solutions of \eqref{CP_disc} to questions about the characteristic function of the system. As will become apparent in Section~\ref{sec:ex} below, the assumption that a characteristic function should exist is less restrictive than it may appear and in particular is satisfied in a number of important examples. In particular, it is straightforward to show that a characteristic function exists whenever  $\rank(T_1)=1$. 
\begin{rem}
A standard argument involving Neumann series shows that  if \eqref{char} holds then for $|\lambda|>\|T_0\|$ we have
$$|\phi_T(\lambda)|\le\frac{\|T_1\|}{|\lambda|-\|T_0\|},$$
and in particular $|\phi_T(\lambda)|\to0$ as $|\lambda|\to\infty$. Note also that the set of poles of $\phi_T$ is contained in $\sigma(T_0)$, but the inclusion may be strict. 

\end{rem}

Since the solution of \eqref{CP_disc} is given by $x(n)=T^nx_0$, $n\ge0$, our aim in this section is to investigate the asymptotic properties of the orbits of the operator $T$. In order to prepare the ground for the main result of this section, Theorem~\ref{thm:asymp_disc} below, we begin with a series of preliminary results, the first few of which are taken more or less directly from \cite{PauSei15}. The first result gives a complete description of the spectrum of the operator $T$. 

\begin{thm}
  \label{thm:spec}
Let $X=\ell^p(\CC^m)$ for some $m\in\NN$ and some $p$ satisfying $1\le p\le\infty$, and let $T\in\B(X)$ be as above. Then 
  $$\sigma(T)\setminus\sigma(T_0)=\big\{\lambda\in\rho(T_0):|\phi_T(\lambda)|=1\big\}.$$
  Moreover, the following hold:  
  \begin{itemize}
      \setlength{\itemsep}{0.5ex}
    \item[\textup{(a)}]  If $1\leq p<\infty$, then $\sigma(T)\setminus \sigma(T_0)\subset \sigma_{ap}(T)\setminus \sigma_p(T)$.
\item[\textup{(b)}] If $p=\infty$, then $\sigma(T)\setminus \sigma(T_0)\subset \sigma_p(T)$
  and, given $\lambda\in\sigma(T)\setminus\sigma(T_0)$,
\begin{equation}
\label{eq:eigenspace}
  \Ker(\lambda -T)=\big\{(\phi_T(\lambda)^k x_0):x_0\in \Ran(R(\lambda,T_0)T_1)\big\}.
  \end{equation}
  In particular, $\dim\Ker(\lambda-T)=\rank(T_1)$ for all $\lambda\in\sigma(T)\setminus\sigma(T_0)$.
  \end{itemize}
 Furthermore, for $\lambda\in\sigma(T)\setminus \sigma(T_0)$ the range of $\lambda-T$ is dense in $X$ if and only if $1<p<\infty$.
\end{thm}

\begin{proof}
See  \cite[Theorem~2.3]{PauSei15}.
\end{proof}

\begin{rem}
As observed in \cite[Remark~2.4]{PauSei15}, the points in $\sigma(T_0)$ can lie either in $\sigma(T)$ or outside it.
\end{rem}

The next result  establishes a useful bound for the norm of the resolvent operator in the neighbourhood of singular points. 

\begin{prp}
  \label{prp:res}
Fix $1\le p\le\infty$ and $m\in\NN$. If $\lambda\in\rho(T_0)$ is such that $|\phi_T(\lambda)|\neq 1$, then
$$  \left|\|R(\lambda,T)\|-\frac{\|R(\lambda,T_0)T_1R(\lambda,T_0)\|}{|1-|\phi_T(\lambda)||}\right|\le\|R(\lambda,T_0)\|.$$
  In particular, for $\lambda_0\in\rho(T_0)$ such that $|\phi_T(\lambda_0)|=1$ we have 
$$\|R(\lambda,T)\|\asymp \frac{1}{|1- |\phi_T(\lambda)||}$$
  as $\lambda\to\lambda_0$ in the region $\{\lambda\in\rho(T_0):|\phi_T(\lambda)|\neq 1\}$.
\end{prp}

\begin{proof}
See  \cite[Proposition~2.5]{PauSei15}. 
\end{proof}

The most important consequence of this result for our present purposes is the following observation. Following \cite{PauSei15}, we call the even integer $n=n_T$ appearing in this result the \emph{resolvent growth parameter}. Here we let 
$$\Omega_T=\big\{\lambda\in\rho(T_0):|\phi_T(\lambda)|=1\big\}.$$

\begin{lem}\label{lem:even}
  Fix $1\le p\le\infty$ and $m\in\NN$, and suppose  $0\in\Omega_T\subset\DD\cup\{1\}$.   Then
  there exists an even integer $n$ with $2\le n\le 2m$ such that $
  1-|\phi_T(e^{i\theta})|\asymp |\theta|^n$
  as $\theta\to 0$.
\end{lem}

\begin{proof}
The result follows by an argument analogous to the proof of \cite[Lemma~2.6]{PauSei15} by considering polynomials in $\sin\theta$ and $\cos\theta$, and by using the fact that $\sin\theta\sim \theta$ as $\theta\to0$.
\end{proof}

We now restrict our attention to systems in which $\sigma(T_0)=\{1-\alpha\}$ and the characteristic function $\phi_T$ is of the specific form
\begin{equation}
\label{eq:spec_phi}
  \phi_T(\lambda)=\frac{\alpha^k}{(\lambda-1+\alpha)^k},\quad \lambda\in\CC\setminus\{ 1-\alpha\},
  \end{equation}
where $\alpha\in(0,1)$ and $k\in\NN$ are given constants. As shall become apparent, even this class is large enough to contain many natural applications; see for instance Section~\ref{sec:ex}. It follows from Theorem~\ref{thm:spec} that in this case 
\begin{equation}\label{eq:special_spec}
\sigma(T)\setminus\{1-\alpha\}=\big\{\lambda\in\CC:|\lambda-\alpha+1|=\alpha\big\}
\end{equation}
and that the resolvent growth parameter is $n_T=2$. The next theorem establishes that under the above assumptions the operator $T$ is power-bounded, which is to say that $\sup_{n\ge0}\|T^n\|<\infty$. Note that even though this result is an analogue of \cite[Theorem~3.1 and Lemma~3.2]{PauSei15} the method of proof used there does not, to the knowledge of the authors, transfer  to the  discrete setting considered here.

\begin{prp}\label{prp:pb}
Let $X=\ell^p(\CC^m)$ for some $m\in\NN$ and some $p$ satisfying $1\le p\le\infty$, and let $T\in\B(X)$ be as above. Then $T$ is power-bounded.
\end{prp}

\begin{proof}
We begin by observing that $\sigma(T)\subset\DD\cup\{1\}$ and that, by the same argument as in \cite[Section~2]{PauSei15}, the resolvent operator has the explicit form
$$R(\lambda,T)x=\bigg(R(\lambda,T_0)x_k+\sum_{\ell=0}^\infty\phi(\lambda)^\ell R(\lambda,T_0)T_1R(\lambda,T_0) x_{k-\ell-1}\bigg)$$
for $\lambda\in\CC$ with $|\lambda|>1$ and $x=(x_k)\in X$. Writing $R_\lambda$ as a shorthand for $R(\lambda,T_0)$ when $|\lambda|>1$, it follows from the functional calculus for bounded operators that 
$$T^nx=\bigg(T_0^nx_k+\sum_{\ell=0}^\infty\bigg(\frac{1}{2\pi i}\oint_\Gamma \lambda^n \phi(\lambda)^\ell R_\lambda T_1 R_\lambda\, \dd\lambda\bigg)x_{k-\ell-1}\bigg),\quad n\ge0,$$
 where $\Gamma$ is any piecewise smooth and positively oriented contour containing the closed unit disc in its interior. In particular,
\begin{equation}\label{eq:pb_est}
\|T^n\|\le \|T_0^n\|+\frac{1}{2\pi}\sum_{\ell=0}^\infty\left\|\oint_\Gamma \lambda^n \phi(\lambda)^\ell R_\lambda T_1 R_\lambda\, \dd\lambda\right\|,\quad n\ge0.
\end{equation}
Moreover,  $T_0$ has spectral radius $1-\alpha\in (0,1)$ and in particular is power-bounded, so in order to show that $T$ is power-bounded it remains only to obtain a uniform bound over $n\ge0$ for the series on the right-hand side.

Note first that our assumptions imply the existence of constant $m\times m$ matrices $C_1,\dotsc,C_{2m-2}$ such that
$$R_\lambda T_1 R_\lambda=\frac{1}{(\lambda-1+\alpha)^{2m}}\sum_{j=0}^{2m-2} C_j\lambda^j,\quad \lambda\in\CC\setminus\{ 1-\alpha\},$$
and hence
$$\sum_{\ell=0}^\infty\left\|\oint_\Gamma \lambda^n \phi(\lambda)^\ell R_\lambda T_1 R_\lambda\, \dd\lambda\right\|\le \sum_{j=0}^{2m-2} \|C_j\|\sum_{\ell=0}^\infty\bigg|\oint_\Gamma\frac{\alpha^{k\ell}\lambda^{n+j}}{(\lambda-1+\alpha)^{k\ell+2m}}\dd\lambda\bigg|,$$
for all $n\ge0$. Now fix $j$ with $1\le j\le 2m-2$. Letting $D_\lambda$ denote differentiation with respect to $\lambda$, a simple application of Cauchy's integral formula shows that, for $n\ge0$, 
$$\begin{aligned}
\frac{1}{2\pi}\sum_{\ell=0}^\infty\bigg|\oint_\Gamma\frac{\alpha^{k\ell}\lambda^{n+j}}{(\lambda-1+\alpha)^{k\ell+2m}}&\,\dd\lambda\bigg|=\sum_{\ell=0}^\infty\alpha^{k\ell}\frac{\big|\big[D_\lambda^{k\ell+2m-1}\lambda^{n+j}\big]_{\lambda=1-\alpha}\big|}{(k\ell+2m-1)!}\\
&\!\!\!\!\!\!\le \frac{1}{\alpha^{2m-1}}\sum_{\ell=0}^{n+j}\binom{n+j}{\ell}\alpha^\ell (1-\alpha)^{n+j-\ell},
\end{aligned}$$
and by the binomial theorem the right-hand side equals $\alpha^{-(2m-1)}$. Combining these estimates with \eqref{eq:pb_est} gives
$$\|T^n\|\le \|T_0^n\|+\frac{1}{\alpha^{2m-1}}\sum_{j=0}^{2m-1}\|C_j\|,\quad n\ge0,$$
and hence $T$ is  power-bounded, as required.
\end{proof}

We conclude this section with one of our main results, which gives a detailed description of the asymptotic behaviour of solutions to \eqref{CP_disc}. Note first though that differentiating  \eqref{char} shows that 
$$-T_1R(1,T_0)^2T_1=\phi_T'(1)T_1.$$
For the function $\phi_T$  defined in \eqref{eq:spec_phi} we have $\phi_T'(1)\ne0$, and hence in this case the operator $T_1R(1,T_0)$ restricts to an isomorphism from $\Ran(R(1,T_0)T_1)$ onto $\Ran(T_1)$. In what follows we write $L$ for the inverse of this isomorphism.

Theorem \ref{thm:asymp_disc} below is analogous to \cite[Theorem~4.3]{PauSei15}. Note however that the rates of convergence obtained here are in general slightly sharper than those obtained in \cite{PauSei15} and in particular involve no logarithmic factors. Indeed the rates are optimal in the sense that if  $n^{-1/2}$ were replaced by some $r_n>0$, $n\ge1$, with $r_n=o(n^{-1/2})$ as $n\to\infty$, then the statement would become false. We obtain this strengthened result as a consequence of a theorem due to Dungey \cite{Dun08}. In what follows, if $X=\ell^p(\CC^m)$ for some $m\in\NN$ and some $p$ satisfying $1\le p\le\infty$, we let $S$ denote the right-shift operator defined by $Sx=(x_{k-1})$ for $x=(x_k)\in X$ and we let
$$Y=\left\{x_0\in X:\lim_{n\to\infty}x(n) \mbox{ exists}\right\},$$
where $x(n)$, $n\ge0$, is the solution of \eqref{CP_disc}.

\begin{thm}\label{thm:asymp_disc} 
Let $X=\ell^p(\CC^m)$ for some $m\in\NN$ and some $p$ satisfying $1\le p\le\infty$, and consider the operator $T\in\B(X)$ defined as above. Define the operator $M\in\B(X)$ by $M(x_k)=(T_1 R(1,T_0) x_k),$ and let the operator $L$ and the space $Y$  be  as above.
\begin{enumerate}
\item[\textup{(a)}] Given $x_0\in X$, we have $x_0\in Y$ if and only if there exists $y_0\in\Ran(T_1)$ such that for the constant sequence  $y$ with entry $y_0$ is an element of $X$ and we have
\begin{equation}
\label{eq:Cesaro_disc}
\bigg\|\frac{1}{n}\sum_{k=1}^nS^{k}Mx_0-y\bigg\|_X\to0,\quad n\to\infty.
\end{equation}
Moreover, if this is the case then $\|x(n)-z\|_X\to0$ as $n\to\infty$, where $z\in X$ is the constant sequence with entry $Ly_0$. In particular, $Y=X$ if and only if $1<p<\infty$ and if $1\le p<\infty$  the only possible candidate for $y$ and $z$ is $0$.
\item[\textup{(b)}] If $x_0\in X$ is such that the convergence in \eqref{eq:Cesaro_disc} is like $O(n^{-1})$ as $n\to\infty$, then
\begin{equation}\label{eq:rate1}
\|x(n)-z\|_X=O\big(n^{-1/2}\big),\quad n\to\infty,
\end{equation}
where $z$ is as above.
\item[\textup{(c)}] For all $x_0\in X$ we have
\begin{equation}\label{eq:rate2}
\|x(n+1)-x(n)\|_X=O\big(n^{-1/2}\big),\quad n\to\infty.
\end{equation}
\end{enumerate}
Furthermore, the rates in \eqref{eq:rate1} and \eqref{eq:rate2} are optimal.
\end{thm}

\begin{proof}
Observe that $x(n)=T^nx_0$, $n\ge0$, so all of the statements can be understood as statements about the orbits of the operator $T$. We show first that $Y=X_0\oplus X_1$, where $X_0=\Ker(I-T)$ and $X_1$ denotes the closure in $X$ of $\Ran(I-T)$. Indeed, it is clear that $X_0\subset Y$ and that $X_0\cap X_1=\{0\}$. Note also that $T$ is power-bounded by Proposition~\ref{prp:pb} and that by \eqref{eq:special_spec} we have $\sigma(T)\cap\T=\{1\}.$ It follows from the Katznelson-Tzafriri theorem \cite[Theorem~1]{KT86} that $\|T^n(I-T)\|\to0$ as $n\to\infty$, and in particular $\Ran(I-T)\subset Y$. Since $T$ is power-bounded, it follows from a straightforward approximation argument that $X_1\subset Y$. Hence $X_0\oplus X_1\subset Y$. Now suppose conversely that $x_0\in Y$. Then by definition the orbit $T^nx_0$, ${n\ge0}$, converges in norm to a limit, and in particular the Ces\`aro averages 
$$\frac1n\sum_{k=1}^n T^kx_0,\quad n\ge1,$$
 also converge in norm. It follows for instance from \cite[Theorem~1.3 of Section~2.1]{Kre85} that $x_0\in X_0\oplus X_1$. Hence $Y= X_0\oplus X_1$, as required. The above argument also shows that if $x_0\in Y$ then $\|x(n)-Px_0\|_X\to0$ as $n\to\infty$, where $P$ is the projection from $Y$ onto $X_0$ along $X_1$.

In order to obtain the quantified statements, observe  that by \cite[Theorem~1.2]{Dun08} we have $\|T^n(I-T)\|=O(n^{-1/2})$ as $n\to\infty$ if and only if the operator $Q \in\B(X)$ defined by 
 $$Q=\frac{T-\beta }{1-\beta}$$
 is power-bounded for some $\beta\in(0,1)$. However, a simple calculation shows that since the characteristic function of $T$ is given by 
 $$\phi_T(\lambda)=\frac{\alpha^k}{(\lambda-1+\alpha)^k},\quad \lambda\in\CC\setminus\{ 1-\alpha\},$$
 for some $\alpha\in(0,1)$ and $k\in\NN$, the operator $Q$ has an associated characteristic function given by 
 $$\phi_Q(\lambda)=\frac{\gamma^k}{(\lambda-1+\gamma)^k},\quad \lambda\in\CC\setminus\{ 1-\gamma\},$$
 where $\gamma=\alpha/(1-\beta)$. Thus if $\beta\in(0,1-\alpha)$, or equivalently $\gamma\in(0,1)$, it follows from another application of Proposition~\ref{prp:pb} that $Q$ is power-bounded, and hence $\|T^n(I-T)\|=O(n^{-1/2})$ as $n\to\infty$. In particular, $\|T^nx\|_X=O(n^{-1/2})$ as $n\to\infty$ for all $x\in \Ran(I-T)$. Optimality of this rate follows straightforwardly  from \cite[Theorem~2.4]{Sei14}, since by Proposition~\ref{prp:res}, Lemma~\ref{lem:even} and the fact that the resolvent growth parameter is $n_T=2$ we have $\|R(e^{i\theta},T)\|\asymp|\theta|^{-2}$ as $\theta\to0$.
 
Thus it remains only to establish the characterisation of the elements of $Y$ and of $X_0\oplus\Ran(I-T)$ in terms of the stated Ces\`aro conditions. However, both of these characterisations follow directly from \cite[Theorem~4.3]{PauSei15}. Indeed, let $A\in\B(X)$ be given by $A=T-I$, so that $Ax=(A_0x_k+A_1x_{k-1})$ for all $x=(x_k)\in X$, where $A_0=T_0-I$ and $A_1=T_1$. A simple calculation shows that $A$ admits a characteristic function $\phi_A$, which satisfies $\phi_A(\lambda)=\phi_T(\lambda+1)$ for all $\lambda\in\CC\setminus\{-\alpha\}$. Noting that $\phi_A(0)=\phi_T(1)=1$, the required statements now follow immediately from the corresponding statements in \cite[Theorem~4.3]{PauSei15}. 
\end{proof}

\begin{rem}
Instead of appealing to Dungey's theorem \cite[Theorem~1.2]{Dun08} in order to obtain a rate of decay for $\|T^n(I-T)\|$ as $n\to\infty$, it would also be possible to use the results in \cite{Sei14, Sei15b} together with Proposition~\ref{prp:res} and Lemma~\ref{lem:even} above in order to obtain similar estimates. These estimates would, however, involve logarithmic terms unless $p=2$, and hence would be weaker than the results obtained by means of Dungey's theorem.
\end{rem}

\section{Examples of discrete-time systems}\label{sec:ex}

In this section we apply Theorem~\ref{thm:asymp_disc} to two simple but illustrative examples. The first of these is the example introduced in \eqref{eq:bugs_intro}. In this case $m=1$, $T_0=1-\alpha$ and $T_1=\alpha$, where $\alpha\in(0,1)$ is a given constant. It is clear that 
$$\sigma(T)=\big\{\lambda\in\CC:|\lambda-1+\alpha|=\alpha\big\},$$
 and a trivial calculation shows that the associated operator $T$ admits the characteristic function
$$\phi_T(\lambda)=\frac{\alpha}{\lambda-1+\alpha},\quad \lambda\in\CC\setminus\{1-\alpha\}.$$
In particular, as a direct consequence of Theorem~\ref{thm:asymp_disc} we obtain the following result about the asymptotic behaviour of solutions to \eqref{eq:bugs_intro}.

\begin{cor}\label{cor:frogs} 
Let $X=\ell^p(\CC)$ for some $p$ satisfying $1\le p\le\infty$, and consider the solution $x(n)$, $n\ge0$, of the system \eqref{eq:bugs_intro} with initial data $x_0\in X$.
\begin{enumerate}
\item[\textup{(a)}] The solution converges to a limit in the norm of $X$ if and only if there exists $c\in \CC$ such that for the constant sequence  $y$ with entry $c$ is an element of $X$ and we have
\begin{equation}
\label{eq:Cesaro_disc_frogs}
\bigg\|\frac{1}{n}\sum_{k=1}^nS^{k}x_0-y\bigg\|_X\to0,\quad n\to\infty.
\end{equation}
Moreover, if this is the case then $\|x(n)-y\|_X\to0$ as $n\to\infty$. In particular, all initial vectors $x_0$ lead to convergence if and only if $1<p<\infty$, and if $1\le p<\infty$ then the only possible candidate for $c$ is 0.
\item[\textup{(b)}] If $x_0\in X$ is such that  the convergence in \eqref{eq:Cesaro_disc_frogs} is like $O(n^{-1})$ as $n\to\infty$, then
\begin{equation}\label{eq:rate1_frogs}
\|x(n)-y\|_X=O\big(n^{-1/2}\big),\quad n\to\infty,
\end{equation}
where $y$ is as above.
\item[\textup{(c)}] For all $x_0\in X$ we have
\begin{equation}\label{eq:rate2_frogs}
\|x(n+1)-x(n)\|_X=O\big(n^{-1/2}\big),\quad n\to\infty.
\end{equation}
\end{enumerate}
Furthermore, the rates in \eqref{eq:rate1_frogs} and \eqref{eq:rate2_frogs} are optimal.
\end{cor}

Note that in the above example the operator $T$ is in fact a contraction. Moreover, the operators $T_0$ and $T_1$ commute in this example, which makes the problem simpler and, at least in principle, makes it possible to use more direct techniques. Our second example is more representative of the general situation in Theorem~\ref{thm:asymp_disc} in the sense that, in general, the operators $T_0$ and $T_1$ do not  commute and the operator $T$ is not a contraction. Indeed, suppose now that $m=2$ and that 
$$T_0
= \left( \begin{array}{cc}
1 & 1 \\
\beta_0 & \beta_1  \end{array} \right)
\quad\mbox{and}\quad
T_1= \left( \begin{array}{cc}
0 & -1 \\
0&0  \end{array} \right),$$
where $\beta_0,\beta_1\in\CC$ are fixed parameters. The corresponding discrete-time system arises for instance if one considers agents whose state vectors $x_k(n)$ are of the form
$$x_k(n)= \left( \begin{array}{c}
d-d_k(n) \\v_k(n)  \end{array} \right),\quad k\in\ZZ,\; n\ge0,$$
where $d_k(n)$ denotes the separation between agent $k$ and agent $k-1$ at time $n$, $d$ is a specified target separation, and $v_k(n)$ is the velocity of agent $k$ at time $n$. This system is precisely of the form \eqref{CP_disc} for the operators $T_0$ and $T_1$ specified above provided that
$${v}_k(n+1)=\beta_0y_k(n)+\beta_1v_k(n),\quad  k\in\ZZ,\; n\ge0,$$
where $y_k(n)=d-d_k(n)$ is the discrepancy between the target separation and the actual separation  of agents $k$ and $k-1$ at time $n$. Thus the parameters $\beta_0, \beta_1$ can be understood as feedback control parameters which should be chosen so as to make the resulting system stable in an appropriate sense. Letting $\alpha_0=-\beta_0$ and $\alpha_1=1-\beta_1$, we observe that the operator $T$ admits the characteristic function $\phi_T(\lambda)=\alpha_0/p(\lambda-1)$ for $p(\lambda-1)\ne0$, where $p(\lambda)=\lambda^2+\alpha_1\lambda+\alpha_0$. Now fix $\alpha_0\in(0,1)$ and set $\alpha_1=2\smash{\alpha_0^{1/2}}$. Then $\sigma(T_0)=\{1-\alpha\}$  and $\phi_T$ is of the form \eqref{eq:spec_phi} with $k=2$ and $\alpha=\smash{\alpha_0^{1/2}}$. A simple application of Theorem~\ref{thm:asymp_disc} now gives the following result. Note that the asymptotic behaviour is determined solely by the initial discrepancies and is independent of the initial velocities.

\begin{cor}\label{cor:disc_contr} 
Let $X=\ell^p(\CC^2)$ for some $p$ satisfying $1\le p\le\infty$, and consider the solution $x(n)$, $n\ge0$, of the system \eqref{CP_disc} for $T$ as above and with initial data $x_0\in X$.
\begin{enumerate}
\item[\textup{(a)}] The solution converges to a limit in the norm of $X$ if and only if there exists $c\in \CC$ such that for the constant sequence  $y$ with entry $c$ is an element of $\ell^p(\CC)$ and we have
\begin{equation}
\label{eq:Cesaro_disc_disc_contr}
\bigg\|\frac{1}{n}\sum_{k=1}^nS^{k}y_0-y\bigg\|_{\ell^p(\CC)}\!\!\!\to0,\quad n\to\infty,
\end{equation}
where $y_0=(y_k(0))\in\ell^p(\CC)$ is the vector of initial discrepancies. Moreover, if this is the case then $\|x(n)-z\|_X\to0$ as $n\to\infty$, where $z\in X$ is given by
$$z=\left( \begin{array}{c}\dots, 
\left( \begin{array}{c}
c \\-\alpha c/2  \end{array} \right),
\left( \begin{array}{c}
c \\-\alpha c/2  \end{array} \right),\dots
 \end{array} \right).$$
 In particular, all initial vectors $x_0$ lead to convergence if and only if $1<p<\infty$, and if $1\le p<\infty$ then the only possible candidate for $c$ is 0.
\item[\textup{(b)}] If $y_0\in\ell^p(\CC)$ is such that  the convergence in \eqref{eq:Cesaro_disc_disc_contr} is like $O(n^{-1})$ as $n\to\infty$, then
\begin{equation}\label{eq:rate1_disc_contr}
\|x(n)-z\|_X=O\big(n^{-1/2}\big),\quad n\to\infty,
\end{equation}
where $z$ is as above.
\item[\textup{(c)}] For all $x_0\in X$ we have
\begin{equation}\label{eq:rate2_disc_contr}
\|x(n+1)-x(n)\|_X=O\big(n^{-1/2}\big),\quad n\to\infty.
\end{equation}
\end{enumerate}
Furthermore, the rates in \eqref{eq:rate1_frogs} and \eqref{eq:rate2_frogs} are optimal.
\end{cor}

More sophisticated models of this type, perhaps involving acceleration, could be handled in an analogous way. We return indirectly to one particular such example in Section~\ref{sec:cont_ex} below.

\section{From discrete to  continuous time}\label{sec:cont}

In this section we turn to the continuous-time analogue of \eqref{CP_disc} studied in \cite{PauSei15} and in particular we show how Theorem~\ref{thm:asymp_disc} can be used to sharpen the main result \cite[Theorem~4.3]{PauSei15} in an important special case. Indeed, let $X=\ell^p(\CC^m)$ for some $m\in\NN$ and some $p$ satisfying $1\le p\le\infty$, and consider the abstract Cauchy problem
\begin{equation}\label{CP_cont}
\begin{cases}
\dot{x}(t)= Ax(t), \quad t\ge0,\\
 x(0)=x_0\in X,
\end{cases}
\end{equation}
where $Ax=(A_0x_k+A_1x_{k-1})$ for all $x=(x_k)\in X$ for suitable $m\times m$ matrices $A_0$ and $A_1$. Here we assume that $A_1\ne 0$ and, analogously to our treatment of the discrete-time system, that the operator $A$ admits a characteristic function $\phi_A$ satisfying
$$A_1R(\lambda,A_0)A_1=\phi_A(\lambda)A_1,\quad \lambda\in\rho(A_0).$$
Systems of this type, and in particular the asymptotic behaviour of solutions of \eqref{CP_cont}, are studied in detail in \cite{PauSei15}. We restrict ourselves here to the important special case in which $\sigma(A_0)=\{-\zeta\}$ and the characteristic function has the  form 
\begin{equation}\label{eq:char_spec}
\phi_A(\lambda)=\frac{\zeta^k}{(\lambda+\zeta)^k},\quad \lambda\in\CC\setminus\{-\zeta\},
\end{equation}
for some $\zeta>0$ and some $k\in\NN$. It follows from the general results in \cite{PauSei15} that in this case 
$$\sigma(A)\setminus\{-\zeta\}=\big\{\lambda\in\CC:|\lambda+\zeta|=\zeta\big\},$$
and moreover the semigroup generated by the operator $A$ is uniformly bounded. From these facts together with certain resolvent bounds analogous to Proposition~\ref{prp:res}, the authors obtained in \cite[Theorem~4.3]{PauSei15} an asymptotic result for the solutions of \eqref{CP_cont}. The proof of this result relied on recent results in the theory of non-uniform stability of $C_0$-semigroups. We now present an improved version of this result for the special case introduced above, where the characteristic function has the form given in \eqref{eq:char_spec}. The proof uses Theorem~\ref{thm:asymp_disc} for the discrete-time setting together  Dungey's result  \cite[Theorem~1.2]{Dun08}, which also establishes a connection between the rates in the continuous and the discrete setting. Note that, by the same argument as in the discrete setting, there exists an isomorphism $L$ mapping $\Ran(A_1)$ onto $\Ran(A_0^{-1}A_1)$ which is the inverse of the restriction of $A_1A_0^{-1}$ to $\Ran(A_0^{-1}A_1)$. We also define
$$Y=\left\{x_0\in X:\lim_{t\to\infty} x(t)\;\mbox{exists}\right\},$$
where $x(t)$, $t\ge0$, is the solution of \eqref{CP_cont}

\begin{thm}\label{thm:asymp_cont} 
Let $X=\ell^p(\CC^m)$ for some $m\in\NN$ and some $p$ satisfying $1\le p\le\infty$, and consider the operator $A\in\B(X)$ defined as above. Define the operator $M\in\B(X)$ by $M(x_k)=(A_1 A_0^{-1} x_k),$ and let the operator $L$ and the space $Y$  be  as above.
\begin{enumerate}
\item[\textup{(a)}] Given $x_0\in X$, we have $x_0\in Y$ if and only if there exists $y_0\in\Ran(A_1)$ such that for the constant sequence  $y$ with entry $y_0$ is an element of $X$ and we have
\begin{equation}
\label{eq:Cesaro_cont}
\bigg\|\frac{1}{n}\sum_{k=1}^nS^{k}Mx_0-y\bigg\|_X\!\!\!\to0,\quad n\to\infty.
\end{equation}
Moreover, if this is the case then $\|x(t)-z\|_X\to0$ as $t\to\infty$, where $z\in X$ is the constant sequence with entry $Ly_0$. In particular, $Y=X$ if and only if $1<p<\infty$ and if $1\le p<\infty$  the only possible candidate for $y$ and $z$ is $0$.
\item[\textup{(b)}] If $x_0\in X$ is such that the convergence in \eqref{eq:Cesaro_cont} is like $O(n^{-1})$ as $n\to\infty$, then
\begin{equation}\label{eq:rate1_cont}
\|x(t)-z\|_X=O\big(t^{-1/2}\big),\quad t\to\infty,
\end{equation}
where $z$ is as above.
\item[\textup{(c)}] For all $x_0\in X$ we have
\begin{equation}\label{eq:rate2_cont}
\|\dot{x}(t)\|_X=O\big(t^{-1/2}\big),\quad t\to\infty.
\end{equation}
\end{enumerate}
Furthermore, the rates in \eqref{eq:rate1_cont} and \eqref{eq:rate2_cont} are optimal.
\end{thm}

\begin{proof}
All of the statements are contained in \cite[Theorem~4.3]{PauSei15}, except for the sharper rates of convergence without logarithms, which require us to show that the $C_0$-semigroup $(T(t))_{t\ge0}$ generated by $A$ satisfies $\|AT(t)\|=O(t^{-1/2})$ as $t\to\infty$.  For $\varepsilon>0$, let $T_\varepsilon=\varepsilon A+I$. A straightforward calculation shows that $T_\varepsilon$ admits the characteristic function
$$\phi_{T_\varepsilon}(\lambda)=\phi_A\big(\varepsilon^{-1}(\lambda-1)\big)=\bigg(\frac{\varepsilon\zeta}{\lambda-1+\varepsilon\zeta}\bigg)^k,\quad \lambda\in\CC\setminus\{1-\varepsilon\zeta\}.$$
Now fix $\varepsilon_0\in(0,\zeta^{-1})$ and let $\alpha=\varepsilon_0\zeta$ and $T=T_{\varepsilon_0}$. Then $T$ has characteristic function of the form given in \eqref{eq:spec_phi} and hence Theorem~\ref{thm:asymp_disc} implies that $\|T^n(I-T)\|=O(n^{-1/2})$ as $n\to\infty$. By the implication (i)$\implies$(v) of \cite[Theorem~1.2]{Dun08} (with $n=1$) we see that the $C_0$-semigroup $(T_{\varepsilon_0}(t))_{t\ge0}$ generated by $A_{\varepsilon_0}=\varepsilon_0 A$ satisfies $\| A_{\varepsilon_0}T_{\varepsilon_0}(t)\|=O(t^{-1/2})$ as $t\to\infty$, from which the result follows immediately.
\end{proof}

\begin{rem}
By removing the logarithmic factor present in \cite[Theorem~4.3]{PauSei15} in the above special case, Theorem~\ref{thm:asymp_cont} gives a partial answer to the question raised in \cite[Remark~4.14(b)]{PauSei15}. It remains open whether, as was conjectured in \cite[Remark~4.14(b)]{PauSei15}, the logarithmic factor can always be removed in the more general setting of \cite[Theorem~4.3]{PauSei15},  where the resolvent growth parameter may be different from 2, but the above result certainly makes this seem  plausible.
\end{rem}

\section{An example in continuous time}\label{sec:cont_ex}

In this final section we return to an important example studied for isntance in \cite{PloSch11, SwaHed96,ZwaFir13}, the so-called \emph{platoon model}. This can be viewed as a continuous-time analogue of the system considered in Corollary~\ref{cor:disc_contr} but with agents' accelerations taken into account as well as their positions and velocities. Specifically, the state vector of agent $k$ now takes to form
$$x_k(t)= \left( \begin{array}{c}
y_k(t) \\
v_k(t)-v\\
a_k(t)  \end{array} \right),\quad k\in\ZZ,\; t\ge0,$$
where $y_k(t)=d_k-d_k(t)$ denotes the discrepancy between the \emph{agent-specific} target separation $d_k$ between agents $k$ and $k-1$ and their actual distance $d_k(t)$  at time $t$, $v_k(t)$ is the velocity of agent $k$ at time $t$, $v$ the target velocity of the entire platoon, and $a_k(t)$ is the acceleration of agent $k$ at time $t$. In particular, we now have $m=3$ in \eqref{CP_cont}. For details on the general platoon model, see for instance \cite[Section~5]{PauSei15} and also \cite{PloSch11, SwaHed96,ZwaFir13}. We restrict ourselves here to the particular case in which the matrices $A_0$ and $A_1$ are given by 
$$A_0
= \left( \begin{array}{ccc}
 0&1& 0 \\
 0 & 0 & 1\\
-\alpha_0 & -\alpha_1 & -\alpha_2  \end{array} \right)
\quad\mbox{and}\quad
A_1= \left( \begin{array}{ccc}
0 & -1 & 0 \\
0&0 & 0\\
0&0 & 0 \end{array} \right)$$
with $\alpha_0=\zeta^3$, $\alpha_1=3\zeta^2$ and $\alpha_2=3\zeta$ for some $\zeta>0$. Then $\sigma(A_0)=\{-\zeta\}$ and  the corresponding operator $A$ admits a characteristic function of the form given in \eqref{eq:char_spec} with $k=3$. The following result, obtained here as an immediate consequence of Theorem~\ref{thm:asymp_cont}, is an improved version of \cite[Theorem~5.1]{PauSei15}.

\begin{cor}\label{cor:plat} 
Let $X=\ell^p(\CC^3)$ for some $p$ satisfying $1\le p\le\infty$, and consider the solution $x(t)$, $t\ge0$, of the system \eqref{CP_cont} for $A$ as above and with initial data $x_0\in X$.
\begin{enumerate}
\item[\textup{(a)}] The solution converges to a limit in the norm of $X$ if and only if there exists $c\in \CC$ such that for the constant sequence  $y$ with entry $c$ is an element of $\ell^p(\CC)$ and we have
\begin{equation}
\label{eq:Cesaro_plat}
\bigg\|\frac{1}{n}\sum_{k=1}^nS^{k}y_0-y\bigg\|_{\ell^p(\CC)}\!\!\!\to0,\quad n\to\infty,
\end{equation}
where $y_0=(y_k(0))\in\ell^p(\CC)$ is the vector of initial discrepancies. Moreover, if this is the case then $\|x(t)-z\|_X\to0$ as $t\to\infty$, where $z\in X$ is given by
$$z=\left( \begin{array}{c}\dots, 
\left( \begin{array}{c}
c \\-\zeta c/3\\0  \end{array} \right),
\left( \begin{array}{c}
c \\-\zeta c/3\\0  \end{array} \right),\dots
 \end{array} \right).$$
 In particular, all initial vectors $x_0$ lead to convergence if and only if $1<p<\infty$, and if $1\le p<\infty$ then the only possible candidate for $c$ is 0.
\item[\textup{(b)}] If $y_0\in\ell^p(\CC)$ is such that the convergence in \eqref{eq:Cesaro_plat} is like $O(n^{-1})$ as $n\to\infty$, then
\begin{equation}\label{eq:rate1_plat}
\|x(t)-z\|_X=O\big(t^{-1/2}\big),\quad t\to\infty,
\end{equation}
where $z$ is as above.
\item[\textup{(c)}] For all $x_0\in X$ we have
\begin{equation}\label{eq:rate2_plat}
\|\dot{x}(t)\|_X=O\big(t^{-1/2}\big),\quad t\to\infty.
\end{equation}
\end{enumerate}
Furthermore, the rates in \eqref{eq:rate1_plat} and \eqref{eq:rate2_plat} are optimal.
\end{cor}

\bibliography{robots-reference}
\bibliographystyle{plain}
\end{document}